\documentclass[a4paper, 11pt]{amsart}   
\usepackage{mathptmx, amssymb,amscd,latexsym, eulervm}   
\usepackage{amsmath}
\usepackage{amsthm}
\usepackage{mathdots}
\usepackage[pagebackref]{hyperref}
\hypersetup{colorlinks=true,linkcolor=blue,citecolor=blue,urlcolor=blue,breaklinks=true}
\usepackage{color}
\usepackage{setspace}
\usepackage{comment}
\usepackage{tabularx}
\usepackage{amsfonts}
\usepackage{paralist}
\usepackage{aliascnt}
\usepackage[initials, lite]{amsrefs}
\usepackage{amscd}
\usepackage{blkarray}
\usepackage{mathbbol}
\usepackage{setspace}
\usepackage[inner=2.4cm,outer=2.4cm,bottom=3.2cm]{geometry}
\usepackage{tikz, tikz-cd}
\usepackage{calligra,mathrsfs}
\usepackage{tikz}
\usetikzlibrary{matrix}
\usetikzlibrary{arrows,calc}
\allowdisplaybreaks

\BibSpec{collection.article}{%
	+{}  {\PrintAuthors}                {author}
	+{,} { \textit}                     {title}
	+{.} { }                            {part}
	+{:} { \textit}                     {subtitle}
	+{,} { \PrintContributions}         {contribution}
	+{,} { \PrintConference}            {conference}
	+{}  {\PrintBook}                   {book}
	+{,} { }                            {booktitle}
	+{,} { }                            {series}
	+{, vol.} { }                            {volume}
	+{,} { }                            {publisher}
	+{,} { \PrintDateB}                 {date}
	+{,} { pp.~}                        {pages}
	+{,} { }                            {status}
	+{,} { \PrintDOI}                   {doi}
	+{,} { available at \eprint}        {eprint}
	+{}  { \parenthesize}               {language}
	+{}  { \PrintTranslation}           {translation}
	+{;} { \PrintReprint}               {reprint}
	+{.} { }                            {note}
	+{.} {}                             {transition}
	+{}  {\SentenceSpace \PrintReviews} {review}
}
\AtBeginDocument{%
	\def\MR#1{}
}

\newcommand{\Coker}{\normalfont\text{Coker}}

\newcommand{\Spec}{{\normalfont\text{Spec}}}

\newtheorem{theorem}{Theorem}[section]

\newtheorem{headthm}{Theorem}

\newaliascnt{headcor}{headthm}
\newtheorem{headcor}[headcor]{Corollary}
\aliascntresetthe{headcor}

\newaliascnt{headconj}{headthm}

\aliascntresetthe{headconj}

\newaliascnt{corollary}{theorem}
\newtheorem{corollary}[corollary]{Corollary}
\aliascntresetthe{corollary}

\newaliascnt{claim}{theorem}

\aliascntresetthe{claim}

\newaliascnt{lemma}{theorem}
\newtheorem{lemma}[lemma]{Lemma}
\aliascntresetthe{lemma}

\newaliascnt{conjecture}{theorem}

\aliascntresetthe{conjecture}

\newaliascnt{proposition}{theorem}
\newtheorem{proposition}[proposition]{Proposition}
\aliascntresetthe{proposition}

\theoremstyle{definition}
\newaliascnt{definition}{theorem}
\newtheorem{definition}[definition]{Definition}
\aliascntresetthe{definition}

\newaliascnt{notation}{theorem}

\aliascntresetthe{notation}

\newaliascnt{example}{theorem}

\aliascntresetthe{example}

\newtheorem*{example*}{Example}
\newaliascnt{examples}{theorem}

\aliascntresetthe{examples}

\newaliascnt{remark}{theorem}
\newtheorem{remark}[remark]{Remark}
\aliascntresetthe{remark}

\newaliascnt{question}{theorem}

\aliascntresetthe{question}

\newaliascnt{questions}{theorem}

\aliascntresetthe{questions}

\newaliascnt{problem}{theorem}

\aliascntresetthe{problem}

\newaliascnt{construction}{theorem}

\aliascntresetthe{construction}

\newaliascnt{setup}{theorem}

\aliascntresetthe{setup}

\newaliascnt{algorithm}{theorem}

\aliascntresetthe{algorithm}

\newaliascnt{observation}{theorem}

\aliascntresetthe{observation}

\newaliascnt{defprop}{theorem}

\aliascntresetthe{defprop}

\DeclareFontFamily{OT1}{pzc}{}
\DeclareFontShape{OT1}{pzc}{m}{it}{<-> s * [1.100] pzcmi7t}{}
\DeclareMathAlphabet{\mathchanc}{OT1}{pzc}{m}{it}

\def\equationautorefname~#1\null{(#1)\null}
\def\sectionautorefname~#1\null{Section #1\null}
\def\subsectionautorefname~#1\null{\S #1\null}

\newcommand{\im}{\text{im}}

\makeatletter
\@namedef{subjclassname@2020}{
  \textup{2020} Mathematics Subject Classification}
\makeatother

\begin{document}

\title[Boundedness of Geometrically Integral $\varepsilon$-klt Del Pezzo Surfaces]{Boundedness of Geometrically Integral $\varepsilon$-klt Del Pezzo Surfaces}

\author[Shikha Bhutani]{Shikha Bhutani}
\address{Shikha Bhutani\\ Department of Mathematics, Michigan State University,
Wells Hall
619 Red Cedar Rd
East Lansing, MI 48824}\email{bhutani4@msu.edu}

\author[Sudipta Das]{Sudipta Das}
\address{Sudipta Das\\ School of Mathematics,
Tata Institute of Fundamental Research
Dr. Homi Bhabha Road, Colaba
Mumbai 400 005, Maharashtra, India}\email{sudiptad@math.tifr.res.in}

\subjclass[2020]{Primary 14J45, 14E30; Secondary 14G17, 14E05, 14D06.}
\keywords{del Pezzo surfaces, BAB conjecture, positive characteristic, imperfect fields, irregularity, Mori fibre spaces, wild conic bundles, $p$-Fermat hypersurfaces.}
\begin{abstract}
We prove the remaining positive characteristic cases of the Borisov--Alexeev--Borisov conjecture for geometrically integral $\varepsilon$-klt del Pezzo surfaces over arbitrary fields. Bernasconi and Martin established boundedness outside characteristics $2$, $3$, and $5$, and reduced the remaining cases to a uniform bound for the irregularity. We prove this bound, thereby completing the boundedness theorem in all positive characteristics.
\end{abstract}
\maketitle
\section{Introduction}

The Borisov--Alexeev--Borisov conjecture predicts that, for fixed dimension and fixed $\varepsilon>0$, the $\varepsilon$-klt Fano varieties form a bounded family.  The toric case was proved by Borisov--Borisov \cite{BorisovBorisov}, while the surface case was proved by Alexeev and developed further by Alexeev--Mori \cites{AlexeevBoundedness, AlexeevMori}; Birkar proved the conjecture in all dimensions in characteristic zero \cite{BirkarBAB}.  

Over perfect fields, the positive characteristic surface case is known; see \cites{AlexeevBoundedness, AlexeevMori}.
A uniform boundedness statement over algebraically closed fields of characteristic $p>5$ is given in \cite[Lemma~3.1]{CTW}. 
Over imperfect fields, regularity need not imply smoothness and generic fibres of Mori fibre spaces may exhibit inseparable phenomena.
These issues have been studied in \cites{FanelliSchroer, PatakfalviWaldron, JiWaldron, TanakaPathologies, BernasconiTanaka}.  More recently, Bernasconi--Tanaka classified geometrically integral regular del Pezzo surfaces that are not geometrically normal \cite{BernasconiTanakaGeometry}.

A basic imperfect field phenomenon is that irregularity need not vanish.  Schr\"oer constructed normal del Pezzo surfaces and regular weak del Pezzo surfaces with positive irregularity \cite{SchroerIrregularity}. Maddock constructed regular del Pezzo surfaces with $q(X):=h^1(X,\mathcal O_X)>0$ \cite{MaddockIrregularity}, and Tanaka constructed regular surfaces of del Pezzo type with positive irregularity \cite{TanakaPathologies}.  Tanaka proved that geometrically integral regular del Pezzo surfaces over arbitrary fields form a bounded family \cite{TanakaBoundedness}.  For singular del Pezzo surfaces, however, irregularity remains a genuine boundedness issue because it enters the Euler characteristic and hence the Hilbert polynomial. Bernasconi and Martin proved the geometrically integral $\varepsilon$-klt del Pezzo surface case of BAB over arbitrary fields in every positive characteristic except $2$, $3$, and $5$, and reduced the remaining cases to a uniform bound for the irregularity; see \cite[Theorem~1.3 and Remark~5.13]{BernasconiMartin}.  The following theorem provides this bound.

\begin{headthm}\label{thm:main}
Let $X$ be a geometrically integral $\varepsilon$-klt del Pezzo surface over a field $k$ of characteristic $p\in\{2,3,5\}$.  Then
$$
q(X)\leq 1+\frac{1}{\varepsilon}.
$$
\end{headthm}

The dependence only on $\varepsilon$ is what is needed in the boundedness argument of Bernasconi--Martin.

\paragraph{\textbf{Strategy of the proof.}}
The proof begins with the minimal regular resolution $f:Y\to X$ and a $K_Y$-MMP.  Since klt surface singularities are rational, $h^1(\mathcal O)$ is unchanged along the birational steps.   The MMP either ends with a regular del Pezzo surface, covered by Bernasconi--Martin, or with a Mori fibre space $\pi:Z\to B$ over a regular projective curve.  In the latter case, Bernasconi--Tanaka show that every fibre is an irreducible reduced plane conic \cite[Proposition~2.18]{BernasconiTanaka}.  Since $R^1\pi_*\mathcal O_Z=0$, Leray gives $q(Z)=h^1(B,\mathcal O_B)$.  The problem is therefore reduced to bounding the genus of the regular, possibly non smooth, curve $B$.

The genus enters through the formula
$$
K_Z^2=K_{Z/B}^2+8\bigl(1-h^1(B,\mathcal O_B)\bigr).
.$$ 
Indeed, we show that $K_Z^2$ is bounded below by $\frac{-32}{5 \epsilon}$ when the generic fibre is smooth (\ref{prop:kz-lower-bound}) and by $8-\frac{6}{\epsilon}$ when the generic fibre is regular non smooth (\ref{prop:char2-wild}). This reduces the problem to analysing $K_{Z/B}^2$.  When the generic conic is smooth, Tanaka's discriminant  \cite{TanakaDiscriminant} gives that the discriminant is an effective divisor in the class $(\det E)^{-1}$, where $E=\pi_* \omega_{Z/B}^{-1}$ (\ref{lem:discriminant-degree-curve}) and one has $K_{Z/B}^2=\deg E$, (\ref{prop:relative-anticanonical-model}) which implies $K_{Z/B}^2\leq0$.  In characteristics $3$ and $5$ the generic conic is necessarily smooth, so this reduces the problem to characteristic $2$ when the generic fibre is non-smooth but regular (see \cite[Proposition~2.2]{WitaszekCBF}).

In characteristic $2$, we show that $K_{Z/B}^2$ is bounded above by $2g-2$, and that concludes the result.
This is achieved by establishing the morphism  $\theta:(F^*_B E/det(E))^\vee\otimes det(E) \longrightarrow\Omega^1_{B/\mathbb F_p}$  (\ref{thm:p-fermat-degree}), 
this morphism is essentially the dual of the map $\alpha$ in Mori-Saito, 
see \cite{MoriSaito}. We show this map is injective, thus allowing us to embed
$(F^*_B E/det(E))^\vee\otimes det(E) $ in $(\Omega^1_{B/{\mathbb{F}_p}})_{\mathrm{tf}}$,
whose composition with the natural map $(\Omega^1_{B/{\mathbb{F}_p}})_{\mathrm {tf}} \rightarrow (\Omega^1_{B/{k}})_{\mathrm{tf}}$ turns out to be generically rank one. 
For such an embedding of the vector bundle, we show that the degree of such a vector bundle is less than $ 2g-2$; see \ref{prop:differential-defect}.
Expanding on the degree of $(F^*_B E/det(E))^\vee\otimes det(E)$, we obtain an upper bound for $\deg E$, i.e. $\deg E \leq 2g-2$. This completes the sketch of the proof.   

These cases account for all outputs of the MMP and complete the outline of the proof of Theorem~A.

\begin{headcor}[Corollary~\ref{cor:all-positive-char-bab}]
Fix a rational number $\varepsilon>0$.  The class of geometrically integral $\varepsilon$-klt del Pezzo surfaces over fields of positive characteristic is bounded over $\Spec\mathbb Z$.
\end{headcor}

\section{Preliminaries}\label{sec:preliminaries}

For basic terminology on del Pezzo surfaces over imperfect fields, we follow \cite[Section~2]{BernasconiMartin} and \cite[Section~2]{BernasconiTanaka}.

\subsection{Notation and conventions}

Throughout the paper, $k$ denotes a field of characteristic $p>0$, and we write $\bar{k}$ (resp.\ $k^{\mathrm{sep}}$) for an algebraic (resp.\ separable) closure of $k$.
If $k'/k$ is a field extension and $X$ is a $k$ scheme, we write $X_{k'}:=X\times_k k'$.
We say that \( X \) is a \textit{variety over \( k \)}, or a \( k \)\textit{-variety}, if \( X \) is an integral scheme that is separated and of finite type over \( k \). A surface is a \( k \)-variety of dimension \( 2 \).
If $X$ is a projective integral $k$-variety, we set $q(X):=h^1(X,\mathcal O_X)=\dim_k H^1(X,\mathcal O_X)$.

For a coherent sheaf $\mathcal F$ on an integral scheme, we write $\mathcal F_{\mathrm{tor}}$
for its torsion subsheaf and $\mathcal F_{\mathrm{tf}}:=\mathcal F/\mathcal F_{\mathrm{tor}}$
for its torsion free quotient.

We say that a $k$-variety $X$ is \emph{geometrically reduced}, \emph{geometrically normal}, \emph{geometrically integral}, or \emph{geometrically regular} if $X_{\bar{k}}$ is so. We are careful to distinguish between \emph{regular} and \emph{smooth}. A $k$-variety $X$ is smooth at $x$ if and only if it is geometrically regular at $x$, and over an imperfect field a regular variety need not be smooth. A projective morphism $f\colon X\to Y$ of normal schemes is called a \emph{contraction} if $f_*\mathcal O_X=\mathcal O_Y$. By a \emph{Mori fibre space} we mean a projective contraction $f\colon X\to Y$ such that $X$ is $\mathbb Q$-factorial and klt, $\dim Y<\dim X$, the relative Picard number of $f$ is one, and $-K_X$ is $f$-ample.

\subsection{Del Pezzo surfaces and surfaces of del Pezzo type}
We fix the conventions used below and record the compatibility statements needed for resolutions and birational contractions.

\begin{definition}
A \emph{regular del Pezzo surface} over $k$ is a reduced projective regular surface $X$ with $-K_X$ ample and $H^0(X,\mathcal O_X)=k$.
\end{definition}

\begin{definition}
A normal projective surface $X$ over $k$ is an \emph{$\varepsilon$-klt del Pezzo surface} if $H^0(X,\mathcal O_X)=k$,
$X$ is $\varepsilon$-klt, and $-K_X$ is ample.

We say that a normal projective surface $X$ over $k$ is \emph{of del Pezzo type} if $H^0(X,\mathcal O_X)=k$
and there exists an effective $\mathbb Q$-divisor $\Delta$ such that $(X,\Delta)$ is klt and $-(K_X+\Delta)$ is big and nef.
\end{definition}

\begin{remark}\label{rem:two-del-pezzo-type-conventions}
Our definition is slightly different from the standard definition where  $-(K_X+\Delta)$ is
required to be ample instead of big and nef; see \cite[Definition~2.6]{BernasconiTanaka}, but we note that the two definitions are equivalent by choosing a different boundary.
\end{remark}

\begin{proposition}\label{prop:klt-rational-sing}
Let $X$ be the spectrum of a local excellent ring with dualising complex and suppose that $(X,\Delta)$ is a klt surface pair for some $\Delta\ge 0$. Then $X$ has rational and $\mathbb Q$-factorial singularities. In particular, if $X$ and $Y$ are projective klt surfaces that are $k$-birational, then $H^i(X,\mathcal O_X)\cong H^i(Y,\mathcal O_Y)$
for every $i\ge 0$.
\end{proposition}

\begin{proof}
This is \cite[Proposition~2.6]{BernasconiMartin}.
\end{proof}

\section{Reduction to Conic Bundles}\label{sec:reduction}

In this section we reduce the irregularity problem to the geometry of a Mori fibre space. Starting from the minimal regular resolution of $X$, we run the surface MMP. Now birational invariance of $H^1(\mathcal O)$ allows us to work on its terminal output. If the MMP ends over a point, the resulting regular del Pezzo surface is covered by known results. The essential case is therefore a Mori fibre space over a curve, where the surface is a conic bundle and its irregularity is identified with that of the base. We record the field theoretic and numerical properties of this fibration that will be used in the subsequent analysis.

\subsection{MMP reduction}\label{subsec:regular-mmp}

Let $X$ be as in Theorem~\ref{thm:main}, and let
$f: Y\to X$ be the minimal regular resolution. We may assume there exists an effective boundary $\Delta_Y$ with $-(K_Y+\Delta_Y)$ ample.
Thus $K_Y$ has negative intersection with every ample divisor and is not
pseudo-effective. Since $Y$ is regular, it is $\mathbb Q$-factorial and
$(Y,0)$ is klt. As $Y$ is projective over the excellent base $\Spec k$, the surface MMP
\cite[Theorem~1.1]{TanakaMMP} yields a sequence of projective birational
contractions
$
Y=Y_0\to Y_1\to\cdots\to Y_r=Z
$
ending in a Mori fibre space $\pi:Z\to B$; see also
\cite[Theorem~3.6]{BFSZ}.

The surfaces $X$, $Y$, and $Z$ are projective, klt, and $k$ birational, so
Proposition~\ref{prop:klt-rational-sing} gives
$$
H^i(X,\mathcal O_X)\cong H^i(Y,\mathcal O_Y)
\cong H^i(Z,\mathcal O_Z)
$$
for every $i\geq0$. In particular, $q(X)=q(Y)=q(Z)$.

\subsection{Mori fibre spaces over curves}\label{subsec:conic-output}

We now assume that the MMP has fibre type output. 
Assume $\dim B=1$. By the Mori fibre space convention,
$\pi_*\mathcal O_Z=\mathcal O_B$ and $B$ is a normal projective integral
curve, hence regular. The morphism $\pi$ is flat: dominance and integrality
make $\mathcal O_{Z,z}$ torsion free over the discrete valuation ring
$\mathcal O_{B,b}$ for every closed $b\in B$ and $z\in Z_b$, so
\cite[Tag 0539]{StacksProject} applies. Flatness over the generic point is
automatic.

By Subsection~\ref{subsec:regular-mmp}, we have $H^0(Z,\mathcal O_Z)=k$. Since $\pi_*\mathcal O_Z=\mathcal O_B$, we also have $H^0(B,\mathcal O_B)=k$. Thus $\pi:Z\to B$ satisfies the hypotheses of \cite[Proposition~2.18]{BernasconiTanaka}, which shows that for every point $b\in B$, including the generic point, the fibre $Z_b$ is an irreducible reduced plane conic over $\kappa(b)$ and $H^0(Z_b,\mathcal O_{Z_b})=\kappa(b).$
In particular, since every closed fibre is irreducible and reduced, every vertical prime divisor on $Z$ is the reduced fibre over its image.

Thus $\pi$ is a conic bundle in the sense of
\cite[Definition~2.3]{TanakaDiscriminant}. By
\cite[Lemma~2.5]{TanakaDiscriminant},
$R^i\pi_*\mathcal O_Z=0$ for $i>0$. Together with
$\pi_*\mathcal O_Z=\mathcal O_B$, the Leray spectral sequence yields
$
H^1(Z,\mathcal O_Z)\cong H^1(B,\mathcal O_B),
$
and hence $q(Z)=h^1(B,\mathcal O_B).$

\begin{lemma}\label{lem:base-function-field-regular}
In the fibre type situation above, the function field $K(B)$ is a regular extension of $k$.
Equivalently, $k$ is algebraically closed in $K(B)$ and $K(B)/k$ is separably generated.
\end{lemma}

\begin{proof}
Geometric reducedness of $X$ makes $K(Z)=K(X)$ separable over $k$
\cite[Tag 04KS]{StacksProject}, while geometric irreducibility
makes $k$ separably algebraically closed in this field
\cite[Tag 054Q]{StacksProject}. Together these imply that $k$ is
algebraically closed in $K(Z)$. By
\cite[Tag 030P]{StacksProject}, the intermediate extension $K(B)/k$
is separable, and $k$ remains algebraically closed in $K(B)$.
Since $K(B)/k$ is finitely
generated, separability means it is separably generated
\cite[Tag 030O]{StacksProject}. Thus $K(B)/k$ is regular.
\end{proof}

\begin{lemma}\label{lem:geometric-integrality-output}
The terminal regular surface $Z$ is geometrically integral over $k$. If the MMP output is of fibre type $\pi:Z\to B$,
then the base curve $B$ is also geometrically integral over $k$.
\end{lemma}

\begin{proof}
The function field $K(Z)=K(X)$ is regular over $k$ because $X$ is
geometrically integral. In the fibre type case, $K(B)/k$ is also regular
by Lemma~\ref{lem:base-function-field-regular}. For an integral variety,
regularity of its function field implies geometric reducedness and
geometric irreducibility by
\cite[Tag 04KS and Tag 054Q]{StacksProject}, hence geometric
integrality by \cite[Tag 038K]{StacksProject}. This applies to $Z$ and,
when present, $B$.
\end{proof}

\begin{remark}\label{rem:base-not-smooth-warning}
The curve $B$ need not be smooth over $k$, i.e., over an imperfect field, regularity of $B$ and regularity of $K(B)/k$ do not imply geometric regularity.  Accordingly, the characteristic $2$ argument is formulated using the absolute Frobenius of $B$ and absolute differentials over $\mathbb F_2$.
\end{remark}

\begin{corollary}\label{cor:constant-field-degree-bound}
Let $\Gamma\subset Z$ be a horizontal integral curve, and put
$$
m_\Gamma:=[K(\Gamma):K(B)],
\qquad
 d_\Gamma:=[H^0(\Gamma,\mathcal O_\Gamma):k].
$$
Then
$
d_\Gamma\le m_\Gamma.
$
\end{corollary}

\begin{proof}
Let $k_\Gamma:=H^0(\Gamma,\mathcal O_\Gamma)$. Then $k_\Gamma$ is a finite field extension of $k$ contained in $K(\Gamma)$. By Lemma~\ref{lem:base-function-field-regular}, the extension $K(B)/k$ is regular. Hence $K(B)$ and $k_\Gamma$ are linearly disjoint over $k$, and therefore
$
[K(B)k_\Gamma:K(B)]=[k_\Gamma:k]=d_\Gamma.
$
Since $K(B)k_\Gamma\subset K(\Gamma)$, the tower law gives
$m_\Gamma=[K(\Gamma):K(B)k_\Gamma]d_\Gamma.$ Thus, in fact, $d_\Gamma$ divides $m_\Gamma$, and in particular
$d_\Gamma\leq m_\Gamma$. 
\end{proof}
\begin{proposition}\label{prop:kz-lower-bound}
In the fibre type situation, one has
$
K_Z^2\ge -\frac{32}{5\varepsilon}.
$
\end{proposition}
\begin{proof}
Let $\Gamma$ be the non fibre extremal curve of
\cite[Lemma~5.5]{BernasconiMartin}, and set
$$
d_\Gamma:=[H^0(\Gamma,\mathcal O_\Gamma):k],
\qquad
m_\Gamma:=[K(\Gamma):K(B)].
$$
By the same lemma, there is an integer $n\geq0$ such that
$
K_Z\cdot\Gamma=d_\Gamma(n-2)
$
and
$
K_Z^2=
\frac{d_\Gamma}{m_\Gamma^2}
\bigl(8m_\Gamma+4n(1-m_\Gamma)\bigr).
$
Since $Z$ arises from the minimal resolution of the original
$\varepsilon$-klt surface by the $K_Y$-MMP,
\cite[Proof of Theorem~5.6, Claim~5.7]{BernasconiMartin} gives
$
n\leq\frac{2}{\varepsilon}.
$
Set
$
A_\Gamma:=8m_\Gamma+4n(1-m_\Gamma).
$
If $A_\Gamma\geq0$, then the displayed formula gives $K_Z^2\geq0$, and there
is nothing to prove.

Assume that $A_\Gamma<0$. Then $n>2$, since for $n\leq2$ one has
$
A_\Gamma=4n+m_\Gamma(8-4n)\geq0.
$
Hence $K_Z\cdot\Gamma=d_\Gamma(n-2)>0$. Here $Z$ is a regular surface of
del Pezzo type, $H^0(Z,\mathcal O_Z)=H^0(B,\mathcal O_B)=k$,
$\rho(Z)=2$, and $\Gamma$ generates the non fibre extremal ray. Together
with $K_Z\cdot\Gamma>0$, these are the hypotheses of the setup preceding
\cite[Proposition~4.7]{BernasconiTanaka}; hence $m_\Gamma\leq5$.
Corollary~\ref{cor:constant-field-degree-bound} gives $d_\Gamma\leq m_\Gamma$.
Thus, in the case $A_\Gamma<0$, we have in particular
$$
d_\Gamma\leq m_\Gamma\leq5,
$$
a refinement that will be used in Proposition~\ref{prop:char2-wild}.

Since $A_\Gamma<0$ and $d_\Gamma\leq m_\Gamma$,
$$
K_Z^2
=\frac{d_\Gamma}{m_\Gamma^2}A_\Gamma
\geq\frac{1}{m_\Gamma}A_\Gamma
=8-4n\frac{m_\Gamma-1}{m_\Gamma}.
$$
Using $m_\Gamma\leq5$ and $n\leq2/\varepsilon$, we obtain
$
K_Z^2
\geq8-\frac{32}{5\varepsilon}
\geq-\frac{32}{5\varepsilon}.
$
\end{proof}

We shall compare this lower bound with two different upper bounds for $K_Z^2$, according as the generic conic is smooth or regular but non smooth.

\section{Relative Anticanonical Geometry of Conic Bundles}\label{sec:relative-conics}

We first record the characteristic free intersection theoretic formulas for a conic bundle that will be used in both branches of the argument. We then treat the case of a smooth generic conic, where the discriminant gives the required upper bound for $K_{Z/B}^2$.

\subsection{Relative anticanonical geometry}

We first fix the standard relative anticanonical model and record the corresponding canonical bundle formula. Both statements are characteristic free and will also be used in the wild characteristic $2$ case.

\begin{proposition}\label{prop:relative-anticanonical-model}
Let $\pi:Z\to B$ be a flat contraction from a regular projective surface to a
regular projective curve, and assume every fibre is an irreducible reduced conic.
Set $E:=\pi_*\omega_{Z/B}^{-1}$.
Then $E$ is locally free of rank $3$, the relative anticanonical morphism gives a
closed immersion $Z\hookrightarrow \mathbb P_B(E)$,
and, writing $\rho:\mathbb P_B(E)\to B$ and letting $H$ denote the divisor class
of $\mathcal O_{\mathbb P_B(E)}(1)$, one has
$$
Z\sim 2H-\rho^*\det E,
\qquad
K_{Z/B}=-H|_Z.
$$
In particular,
$
K_{Z/B}^2= \deg E.
$
\end{proposition}

\begin{proof}
By \cite[Proposition~2.7]{TanakaDiscriminant}, $E$ is locally free of rank $3$ and $\omega_{Z/B}^{-1}$ is relatively very ample; the associated closed immersion $\iota:Z\hookrightarrow P:=\mathbb P_B(E)$
satisfies $\iota^*\mathcal O_P(1)\simeq \omega_{Z/B}^{-1}$. We use the quotient convention and write $H=c_1(\mathcal O_P(1))$. Since $P$ is regular and $Z\subset P$ has codimension one, $Z$ is an effective Cartier divisor. Since each fibre of $Z\to B$ is a plane conic, the projective bundle formula gives $Z\sim 2H+\rho^*M$ for some $M\in\operatorname{Pic}(B)$; see \cite[Tag 02TX]{StacksProject}.
We determine the base twist $M$ by comparing adjunction with the anticanonical embedding.

Since $\rho$ is a smooth $\mathbb P^2$-bundle, $\omega_{P/B}=\det\Omega^1_{P/B}$, and the relative Euler sequence gives
$$
K_{P/B}=-3H+\rho^*\det E.
$$
Adjunction therefore yields
$$
K_{Z/B}
=
\bigl(-H+\rho^*(\det E\otimes M)\bigr)|_Z.
$$
Here $\rho|_Z=\pi$, so the pullback term is $\pi^*(\det E\otimes M)$.
On the other hand,
$\mathcal O_Z(-K_{Z/B})\simeq \omega_{Z/B}^{-1}\simeq\mathcal O_P(1)|_Z$, so
$K_{Z/B}=-H|_Z$. Since $\pi_*\mathcal O_Z=\mathcal O_B$, pullback on
Picard groups is injective; hence $M\simeq(\det E)^{-1}$. Thus
$
Z\sim 2H-\rho^*\det E.
$

Finally, $H$ restricts to the hyperplane class on every fibre
$P_b\simeq\mathbb P^2_{\kappa(b)}$. Applying
\cite[Tag 02TW]{StacksProject} to $c_1(N)\cap[B]$ and taking degrees
therefore gives $H^2\cdot\rho^*N=\deg N$ for every line bundle $N$ on $B$.
Since $\dim B=1$, the terms involving $c_2(E)$ and $c_3(E)$ vanish in
the Chern relation
\cite[Tag 02U0]{StacksProject}.
Taking degrees and using $c_1(E)=c_1(\det E)$, we obtain
$H^3=H^2\cdot\rho^*\det E=\deg E$.
Together with $K_{Z/B}=-H|_Z$ and $Z\sim2H-\rho^*\det E$, this yields
$$
K_{Z/B}^2
=
H^2\cdot Z
=
2H^3-H^2\cdot\rho^*\det E
=
\deg E.\eqno\hbox{\qedhere}
$$
\end{proof}

\subsection{Smooth generic fibres}

Assume now that the generic conic is smooth. We apply Tanaka's discriminant construction to the relative anticanonical model; the resulting divisor determines the sign of $K_{Z/B}^2$ and hence bounds the genus of $B$.

\begin{lemma}\label{lem:discriminant-degree-curve}
Let $\pi:Z\to B$ be as in Proposition~\ref{prop:relative-anticanonical-model}, and
assume that the generic fibre is smooth over $K(B)$. Put
$E:=\pi_*\omega_{Z/B}^{-1}$. Then the discriminant scheme
$\Delta_\pi$ is an effective Cartier divisor on $B$ and
$
\mathcal O_B(\Delta_\pi)\simeq(\det E)^{-1}.
$
In particular,
$
\deg\Delta_\pi=-\deg E=-K_{Z/B}^2.
$
\end{lemma}

\begin{proof}
By Proposition~\ref{prop:relative-anticanonical-model}, the relative equation is a section
$
q\in
H^0\bigl(B,\operatorname{Sym}^2E\otimes(\det E)^{-1}\bigr).
$
Tanaka's discriminant construction
\cite[Definition~3.4, Proposition~3.3, and Theorem~3.13]{TanakaDiscriminant}
associates to $q$ local principal discriminants which glue and whose vanishing
detects the non smooth fibres. The coordinate change calculation in the proof of \cite[Proposition~3.3]{TanakaDiscriminant} shows that the ternary discriminant has determinant weight $2$; it is cubic in the coefficients. Consequently these local discriminants define a global section
$$
\delta(q)\in
H^0\bigl(B,
(\det E)^2\otimes((\det E)^{-1})^3\bigr)
=
H^0\bigl(B,(\det E)^{-1}\bigr).
$$
The generic fibre is smooth, so $\delta(q)$ is nonzero. Since $B$ is a regular
integral curve, its zero scheme is an effective Cartier divisor, and by
Tanaka's construction it is precisely $\Delta_\pi$. Hence
$
\mathcal O_B(\Delta_\pi)\simeq(\det E)^{-1}.
$
Taking degrees and using
$K_{Z/B}^2=\deg E$ from
Proposition~\ref{prop:relative-anticanonical-model} gives the final identity.
\end{proof}

\begin{remark}\label{rem:discriminant-degree-tanaka}
When the base is smooth projective over an algebraically closed field, the
identity
$
\deg\Delta_\pi=-K_{Z/B}^2
$
also follows from Tanaka's Mori--Mukai formula; see
\cite[Remark~5.11 and Theorem~5.15]{TanakaDiscriminant}. The proof above
instead uses Tanaka's general discriminant construction, together with the
explicit transformation law for the relative anticanonical equation, and
therefore applies to the regular curve $B$, which need not be smooth over $k$.
\end{remark}

\begin{proposition}\label{prop:gensmooth-fibre-bound}
Retain the fibre type MMP setting of
Subsections~\ref{subsec:regular-mmp}--\ref{subsec:conic-output}. Assume that the
generic fibre of $\pi$ is smooth over $K(B)$. Then
$$
q(Z)\le 1+\frac{4}{5\varepsilon}.
$$
\end{proposition}

\begin{proof}
Set $g:=h^1(B,\mathcal O_B).$ By Subsection~\ref{subsec:conic-output}, we have
$q(Z)=g.$ By Lemma~\ref{lem:discriminant-degree-curve}, there is an effective Cartier divisor $\Delta_\pi$ on $B$ such that
$\deg\Delta_\pi=-K_{Z/B}^2$, and hence $K_{Z/B}^2\le 0$. By \cite[Lemma 4.26]{LiuAGAC}, we get that $K_Z\sim K_{Z/B}+\pi^*K_B$. Taking the self intersection we get that $$
K_Z^2=K_{Z/B}^2+2K_{Z/B}\cdot\pi^*K_B
=K_{Z/B}^2+8(1-g) 
$$ where the last term comes from the fact that for every Cartier divisor $D$ on $B$, 
$K_{Z/B}\cdot\pi^*D=-2\deg D$ since fibres are conic and $\deg K_B=2g-2$ .
Therefore $K_Z^2\le 8(1-g)$.
Combining this with Proposition~\ref{prop:kz-lower-bound}, we get
$$
-\frac{32}{5\varepsilon}\le K_Z^2\le 8(1-g).
$$
Hence $g\le 1+\frac{4}{5\varepsilon}$.
Since $q(Z)=g$, the proposition follows.
\end{proof}

This completes the generically smooth branch. In characteristic $2$, however, a regular generic conic may be non smooth and its discriminant then vanishes identically. Sections~\ref{sec:differentials} and~\ref{sec:wild-conics} provide the replacement for the discriminant argument.

\section{Differentials on Regular Curves over Imperfect Fields}\label{sec:differentials}

Let $B$ be a regular projective curve over a field $k$ of characteristic $p>0$. 
Since $B$ need not be smooth over $k$, the relative differentials $\Omega^1_{B/k}$ may have torsion, while the constant part of the absolute differentials need not be saturated.
Regular non smooth curves over imperfect fields and their Frobenius behaviour have been studied systematically in recent work; see, for example, \cite{HilarioStoehr}. We first compare the two differential defects and then derive a degree estimate that is independent of the characteristic and the rank.

\subsection{Absolute and relative differentials}

Regularity of $K(B)/k$ ensures that constant differentials remain independent after extension to $K(B)$ and hence define a torsion free subsheaf on $B$. We quantify the failure of $\Omega^1_{B/k}$ to be a line bundle of degree $2g-2$. Locally, the fundamental class morphism identifies the torsion length of $\Omega^1_{B/k}$ with the colength of its torsion free quotient in $\omega_{B/k}$. The same torsion also controls the failure of the constant differentials to be saturated in the absolute differentials.

\begin{lemma}\label{lem:local-cotangent-defect}
Let $B$ be a regular projective integral curve over a field $k$, and assume
that $K(B)/k$ is separably generated. Fix a closed immersion $i:B\hookrightarrow\mathbb P^N_k$,
and let $\lambda:\Omega^1_{B/k}\longrightarrow\omega_{B/k}$
be the fundamental class morphism associated with the factorisation $B\hookrightarrow\mathbb P^N_k\longrightarrow\Spec k$.
Let $b\in B$ be a closed point, set $A:=\mathcal O_{B,b}$.
Then $\lambda_b$ annihilates $\left(\Omega^1_{A/k}\right)_{\mathrm{tor}}$ and induces an injection $$\overline\lambda_b:\Omega^1_{A/k}/\left(\Omega^1_{A/k}\right)_{\mathrm{tor}}\hookrightarrow\omega_{B/k,b}$$
such that
$
\operatorname{length}_A\operatorname{coker}(\overline\lambda_b)
=
\operatorname{length}_A \left(\Omega^1_{A/k}\right)_{\mathrm{tor}}.
$
\end{lemma}

\begin{proof}
By \cite[Tag 0E9K]{StacksProject}, the structure morphism $B\to\Spec k$ is a local complete intersection. For the fixed factorisation through $\mathbb P^N_k$, \cite[Tag 0E9J]{StacksProject} shows that the immersion $i$ is regular. Localising the conormal sequence at $b$ therefore gives a presentation
$$
A^r\xrightarrow{\Phi}A^{r+1}
\longrightarrow\Omega^1_{A/k}\longrightarrow0.
$$
After tensoring with $K:=K(B)$, the cokernel is the one dimensional
$K$ vector space $\Omega^1_{K/k}$, so $\Phi\otimes_AK$ has rank $r$.
Since $A$ is a DVR, $\Phi$ is injective. Choose a uniformiser $\varpi$ of $A$. By Smith normal form, there are bases
$e_1,\ldots,e_r$ and $f_1,\ldots,f_{r+1}$ such that
$
\Phi(e_i)=\varpi^{a_i}f_i,
\text{ for } a_i\geq0.
$
Hence
$$
\Omega^1_{A/k}
\simeq
Af_{r+1}\oplus\bigoplus_{i=1}^r A/(\varpi^{a_i})f_i,
$$
and therefore
$
\Omega^1_{A/k}/\left(\Omega^1_{A/k}\right)_{\mathrm{tor}}\simeq Af_{r+1},
\text{ and }
\operatorname{length}_A \left(\Omega^1_{A/k}\right)_{\mathrm{tor}}=\sum_{i=1}^r a_i.
$

By the determinant description of the fundamental class morphism
\cite[Tag 0E9Y,Tag 0E9Z]{StacksProject}, in these bases $\lambda_b$
annihilates the torsion summands and sends $f_{r+1}$ to
$\varpi^{a_1+\cdots+a_r}$ times a generator of $\omega_{B/k,b}$. Thus
$\overline\lambda_b$ is multiplication by $\varpi^{a_1+\cdots+a_r}$ between
free rank one $A$-modules. It is injective, and
$$
\operatorname{length}_A\operatorname{coker}(\overline\lambda_b)
=
\sum_{i=1}^r a_i
=
\operatorname{length}_A (\left(\Omega^1_{A/k}\right)_{\mathrm{tor}}).
$$
\end{proof}

\begin{proposition}\label{prop:global-differential-defect}
Let $B$ be a regular projective integral curve over a field $k$ of characteristic $p>0$, and assume that $K(B)/k$ is regular.
Then the fundamental class morphism induces an exact sequence
$$
0\longrightarrow \left(\Omega^1_{B/k}\right)_{\mathrm{tf}}
\longrightarrow\omega_{B/k}
\longrightarrow\mathcal D_{B/k}
\longrightarrow0,
$$
where $\mathcal D_{B/k}$ is a coherent torsion sheaf satisfying $\deg\mathcal D_{B/k}=\deg \left(\Omega^1_{B/k}\right)_{\mathrm{tor}}$.

Consequently, $\deg \left(\Omega^1_{B/k}\right)_{\mathrm{tf}}=2g-2-\deg \left(\Omega^1_{B/k}\right)_{\mathrm{tor}}$.
Let $C\subset \left(\Omega^1_{B/\mathbb F_p}\right)_{\mathrm{tf}}$ be the image of $$\Omega^1_{k/\mathbb F_p}\otimes_k\mathcal O_B\rightarrow \left(\Omega^1_{B/\mathbb F_p}\right)_{\mathrm{tf}},
\text{ and put } S:=\ker\left(\left(\Omega^1_{B/\mathbb F_p}\right)_{\mathrm{tf}}\rightarrow \left(\Omega^1_{B/k}\right)_{\mathrm{tf}}\right).$$
Then $S/C$ is a coherent torsion sheaf and
$
\deg(S/C)\leq \deg \left(\Omega^1_{B/k}\right)_{\mathrm{tor}}.$
\end{proposition}

\begin{proof}
Since $K(B)/k$ is regular, it is separably generated, so $\Omega^1_{B/k}$
has generic rank one. Thus $\left(\Omega^1_{B/k}\right)_{\mathrm{tf}}$ is a line bundle. Fix a projective embedding
of $B$ and the associated fundamental class morphism
$\lambda:\Omega^1_{B/k}\longrightarrow\omega_{B/k}$. Since $\omega_{B/k}$
is torsion free, $\lambda$ factors through $\left(\Omega^1_{B/k}\right)_{\mathrm{tf}}$, and
Lemma~\ref{lem:local-cotangent-defect} gives an injection
$
\left(\Omega^1_{B/k}\right)_{\mathrm{tf}}\hookrightarrow\omega_{B/k}
$
whose cokernel $\mathcal D_{B/k}$ is a coherent torsion sheaf with
$\deg\mathcal D_{B/k}=\deg \left(\Omega^1_{B/k}\right)_{\mathrm{tor}}$.

Because $B$ is proper and integral, $H^0(B,\mathcal O_B)$ is a finite
extension of $k$ contained in $K(B)$. Regularity of $K(B)/k$ therefore gives
$H^0(B,\mathcal O_B)=k$, so $\chi(\mathcal O_B)=1-g$. Since $B$ is regular,
it is Gorenstein, and Riemann--Roch \cite[Tag 0BS6]{StacksProject} gives $\deg\omega_{B/k}=-2\chi(\mathcal O_B)=2g-2.$
Taking degrees yields $\deg \left(\Omega^1_{B/k}\right)_{\mathrm{tf}}=2g-2-\deg \left(\Omega^1_{B/k}\right)_{\mathrm{tor}}$.

Now we show $
\deg(S/C)\leq \deg \left(\Omega^1_{B/k}\right)_{\mathrm{tor}}.$ Consider the commutative diagram: 

\[
\begin{tikzcd}
0
  \arrow[r]
&
\Omega^1_{K/\mathbb{F}_p}\otimes \mathcal{O}_B
  \arrow[r, "I"]
  \arrow[d, "u"]
&
\Omega^1_{B/\mathbb{F}_p}
  \arrow[r, "Q"]
  \arrow[d, "v"]
&
\Omega^1_{B/k}
  \arrow[r]
  \arrow[d, "w"]
&
0
\\
0
  \arrow[r]
&
\ker q
  \arrow[r, "i"]
&
\bigl(\Omega^1_{B/\mathbb{F}_p}\bigr)_{\mathrm{tf}}
  \arrow[r, "q"]
&
\bigl(\Omega^1_{B/k}\bigr)_{\mathrm{tf}}
  \arrow[r]
&
0
\end{tikzcd}
\]
Here, the first sequence is the cotangent sequence. The injectivity of $I$ follows from the fact that $k(B)/k$ is regular, making $I$ generically injective. 
The kernel vanishes generically and is a torsion subsheaf of the torsion-free sheaf $\Omega^1_{k/\mathbb F_p}\otimes_k\mathcal O_B$  so is zero. The maps $v$ and $w$ are natural maps to the quotient sheaf, and $u=v \circ I$, indeed, note that $q \circ v \circ I=w \circ Q \circ I=0$. This implies $\im v \circ I \subset \ker q.$ So, $u$ is the inclusion map. Now we apply the Snake Lemma to the diagram to get the exact sequence: 
$$0 \rightarrow \ker u \rightarrow \ker v \rightarrow \ker w\rightarrow \Coker  u \rightarrow \Coker v \rightarrow \Coker w \rightarrow 0$$  Note that $\ker v = \bigl(\Omega^1_{B/\mathbb{F}_p}\bigr)_{\mathrm{tor}}$, $\ker w= \bigl(\Omega^1_{B/k}\bigr)_{\mathrm{tor}}$, $\Coker v =\Coker w=0$ since $v$ and $w$ are surjective maps and $\Coker u= \ker q/ \im (u= v \circ I)=S/C$. So we obtain the following sequence: 
$$\bigl(\Omega^1_{B/\mathbb{F}_p}\bigr)_{\mathrm{tor}} \rightarrow \bigl(\Omega^1_{B/k}\bigr)_{\mathrm{tor}} \rightarrow S/C \rightarrow 0$$

Thus $S/C$ is a quotient of the coherent torsion sheaf $\left(\Omega^1_{B/k}\right)_{\mathrm{tor}}$, so it is coherent
and $\deg(S/C)\leq \deg \left(\Omega^1_{B/k}\right)_{\mathrm{tor}}$.
\end{proof}

\begin{example*}
The two defects above occur simultaneously even for a regular projective geometrically integral curve. Let
$k=\mathbb F_2(b)$ and
$$
B:=V\bigl(Y^2Z+bZ^3+X^3\bigr)\subset\mathbb P^2_k.
$$
After base change to $\overline{k}$ and replacing $Y$ by $Y+\sqrt b\,Z$, the equation becomes
$Y^2Z+X^3=0$, so $B$ is geometrically integral. The point at infinity is smooth, while on the chart $Z=1$ the unique non smooth point is the closed point $\mathfrak p=(x)$, with residue field $k(\sqrt b)$. Its local ring
$A:=\mathcal O_{B,\mathfrak p}$ is a DVR with uniformiser $x$, so $B$ is regular.

Relatively over $k$, the equation gives $x^2dx=0$, and hence
$
\Omega^1_{A/k}\simeq A\,dy\oplus A/(x^2)\,dx.
$
Thus $$\left(\left(\Omega^1_{B/k}\right)_{\mathrm{tor}}\right)_{\mathfrak p}\simeq A/(x^2)\,dx \text{ and } \left(\left(\Omega^1_{B/k}\right)_{\mathrm{tf}}\right)_{\mathfrak p}\simeq A\,dy.$$ Absolutely over $\mathbb F_2$ one has
$db=x^2dx$, so
$$
(\left(\Omega^1_{B/\mathbb{F}_p}\right)_{\mathrm{tf}})_{\mathfrak p}\simeq A\,dx\oplus A\,dy,
\qquad
C_{\mathfrak p}=A\,x^2dx,
\qquad
S_{\mathfrak p}=A\,dx,
$$
and therefore $(S/C)_{\mathfrak p}\simeq A/(x^2).$ The fundamental class map sends $dy$, up to a unit, to $x^2$ times a local generator of $\omega_{B/k}$, so its cokernel is again $A/(x^2)$. Thus the torsion in the relative differentials, the colength in the dualizing line, and the saturation defect all have the same local length. In this example $g=1$ and $[\kappa(\mathfrak p):k]=2$, hence $\deg \left(\Omega^1_{B/k}\right)_{\mathrm{tor}}=4$, $\deg(S/C)=4$, and $\deg \left(\Omega^1_{B/k}\right)_{\mathrm{tf}}=-4=2g-2-\deg \left(\Omega^1_{B/k}\right)_{\mathrm{tor}}$.
\end{example*}

\subsection{A degree estimate}

The preceding comparison is useful beyond the rank two situation arising from a conic. The point is that all but one generic direction of a vector bundle mapping to the absolute differentials are controlled by the constant part, while the remaining direction lies in the relative differential line. The same defect $\deg \left(\Omega^1_{B/k}\right)_{\mathrm{tor}}$ occurs with opposite signs in the two contributions.

\begin{proposition}\label{prop:differential-defect}
Keep the notation and assumptions of Proposition~\ref{prop:global-differential-defect}. Let $\mathcal V$ be a vector bundle of finite rank equipped with an injection 
$
\mathcal V\hookrightarrow \left(\Omega^1_{B/\mathbb F_p}\right)_{\mathrm{tf}}
$, that the composite
$
\mathcal V\hookrightarrow \left(\Omega^1_{B/\mathbb F_p}\right)_{\mathrm{tf}}\longrightarrow
\left(\Omega^1_{B/k}\right)_{\mathrm{tf}}
$
has generic rank one, we have,
$$
\deg\mathcal V\leq2g-2.
$$
\end{proposition}

\begin{proof}
Consider the composition map $\Phi: \mathcal V\longrightarrow \left(\Omega^1_{B/\mathbb F_p}\right)_{\mathrm{tf}}\longrightarrow \left(\Omega^1_{B/k}\right)_{\mathrm{tf}}$, we have 
$$
\deg\mathcal V
=\deg \im \Phi+\deg \ker \Phi$$
Since, $\im \Phi \subseteq \left(\Omega^1_{B/k}\right)_{\mathrm{tf}}$ we get that $\deg \im \Phi \leq \deg \left(\Omega^1_{B/k}\right)_{\mathrm{tf}}=2g-2-\deg \left(\Omega^1_{B/k}\right)_{\mathrm{tor}}$ by previous lemma.
Also note that $\ker \Phi \subseteq \ker\left(\left(\Omega^1_{B/\mathbb F_p}\right)_{\mathrm{tf}}\rightarrow \left(\Omega^1_{B/k}\right)_{\mathrm{tf}}\right)=S$.
At this point, we would like to write $\deg \ker \Phi \leq \deg S$, but the two vector bundles have different ranks, so we cannot compare their degrees; we do a few extra steps to reach the conclusion. Since $\ker \Phi \subseteq S$, we extract the part of $\ker \Phi$ contained in $C$ where $C$ is the 
the image of $\Omega^1_{k/\mathbb F_p}\otimes_k\mathcal O_B\rightarrow \left(\Omega^1_{B/\mathbb F_p}\right)_{\mathrm{tf}}$ as in the previous lemma. We write $C_{\Phi}$ to be $(\ker \Phi) \cap C$. Then $\ker \Phi /C_{\Phi} \hookrightarrow S/C.$ Since $C_{\Phi}$ is coherent, the inclusion $C_{\Phi} \hookrightarrow C$ factors through $C_{\Phi} \hookrightarrow U \otimes_k \mathcal{O}_B$ where $U$ is finite dimesional subspace of $\Omega^1_{k/\mathbb{F}_p}.$ The target space comes from the fact that $C \simeq \Omega^1_{k/\mathbb{F}_p} \otimes_k \mathcal{O}_B$ indeed, by \cite[Tag 0322]{StacksProject}, the natural map
$$
K\otimes_k\Omega^1_{k/\mathbb F_p}
\longrightarrow
\Omega^1_{K/\mathbb F_p}
$$
is injective for any finitely generated regular field extension of characteristic $p>0$, $K/k.$ Consequently, the induced morphism $W\otimes_k\mathcal O_B\longrightarrow \left(\Omega^1_{B/\mathbb F_p}\right)_{\mathrm{tf}}$
is injective for every finite-dimensional $ k$-subspace
$W\subset\Omega^1_{k/\mathbb F_p}$. Indeed, the map is generically injective, and its kernel is torsion, and its locally
free source on the integral curve B is torsion free, so the kernel vanishes and gives that $C \simeq \Omega^1_{k/\mathbb{F}_p} \otimes_k \mathcal{O}_B$. 
We take top exterior powers $C_{\Phi} \hookrightarrow U \otimes_k \mathcal{O}_B$  to get
$$
\det C_{\Phi} \longrightarrow
\bigwedge^{\operatorname{rk} W_0} U \otimes_k \mathcal{O}_B,
$$
whose target is a trivial vector bundle. At least one component is a nonzero map
$
\det C_{\Phi} \longrightarrow \mathcal{O}_B.$
Therefore, $
\deg C_{\Phi} \leq 0$ and 
\begin{equation}\label{eq}
\begin{aligned}
\deg\mathcal V &=\deg \im \Phi+\deg \ker \Phi\\ &\leq 2g-2-\deg \left(\Omega^1_{B/k}\right)_{\mathrm{tor}}+\deg (\ker \Phi)/C_{\Phi} +\deg C_{\Phi}\\
& \leq
2g-2-\deg \left(\Omega^1_{B/k}\right)_{\mathrm{tor}}+\deg S/C+0
\end{aligned}
\end{equation}

Since $\deg S/C \leq \deg \left(\Omega^1_{B/k}\right)_{\mathrm{tor}}$, we conclude $\deg\mathcal V\leq2g-2.$
\end{proof}

\subsection{Families of \texorpdfstring{$p$}{p}-Fermat hypersurfaces}\label{subsec:p-fermat}

The wild conic equation is a $2$-Fermat equation. The same differential mechanism applies to families of $p$-Fermat hypersurfaces. Shepherd--Barron observed that a wild conic
bundle has no discriminant curve and hence no analogue of the usual discriminant
formula \cite[Section~10]{ShepherdBarronFano}. Mori--Saito studied wild
hypersurface bundles in the classification of Fano threefolds
\cite{MoriSaito}. Schr\"oer studies $p$-Fermat hypersurfaces and characterizes their regularity by the $p$-independence of the coefficient field
\cite[Proposition~3.2 and Theorem~3.3]{SchroerFibrations}. More recently,
Tanaka's discriminant theory shows that the ordinary discriminant is useful for
generically smooth conic bundles but becomes less useful in the wild case, where even
the discriminant bundle can be difficult to determine
\cite[Subsection~1.1 and Remark~7.7]{TanakaDiscriminant}.

In characteristic $p$, the $p$-th power monomials define a canonical inclusion $F_B^*E\hookrightarrow\operatorname{Sym}^pE.$ The next theorem gives a numerical consequence for such families over regular curves over imperfect fields.

\begin{theorem}\label{thm:p-fermat-degree}
Let $k$ be a field of characteristic $p>0$, let $B$ be a geometrically integral
regular projective curve over $k$, and put
$
g:=h^1(B,\mathcal O_B).
$
Let $E$ be a vector bundle of rank $n+1$ on $B$, with $n\geq1$, and let $A$ be
a line bundle equipped with a line subbundle
$
s:A\hookrightarrow F_B^*E.
$
Via the canonical Frobenius subbundle
$F_B^*E\hookrightarrow\operatorname{Sym}^pE$, let
$X\subset\mathbb P_B(E)$ be the relative hypersurface defined by $s$, of degree $p$ on each fibre of $\mathbb P_B(E)\to B$. Assume that $X$ is geometrically integral over $k$ and that its generic fibre $X_\eta$ is a regular $p$-Fermat hypersurface. Then
$$
(n+1)\deg A-p\deg E\leq2g-2.
$$
\end{theorem}

\begin{proof}
Put $K:=K(B)$.
Since $A\subset F_B^*E$ is a line subbundle, $F_B^*E/A$ is locally free of rank $n$.
The Frobenius pullback carries its canonical connection
$$
\nabla:F_B^*E
\longrightarrow
F_B^*E\otimes\Omega^1_{B/\mathbb F_p},
\qquad
\nabla(a\otimes e)=da\otimes e.
$$
Restricting $\nabla$ to $A$ and projecting to $F_B^*E/A$ gives the second fundamental
form
$
\beta:A\longrightarrow F_B^*E/A \otimes\Omega^1_{B/\mathbb F_p}.
$
The Leibniz correction term lies in $A$ and disappears after projection, so
$\beta$ is $\mathcal O_B$ linear. By adjunction, it is equivalently a morphism
$$
\theta:(F_B^*E/A)^\vee\otimes A\longrightarrow\Omega^1_{B/\mathbb F_p}.
$$

We first determine the generic rank of $\theta$. Choose a generator $\ell$ of
$A_K$ and a basis $e_0,\ldots,e_n$ of $E_K$. After scaling $\ell$, the generic
equation has the form
$$
x_0^p+u_1x_1^p+\cdots+u_nx_n^p=0,
$$
and
$
s(\ell)=e_0^{[p]}+u_1e_1^{[p]}+\cdots+u_ne_n^{[p]}.
$
Here $e_i^{[p]}$ denotes the Frobenius pullback basis, which is horizontal for the canonical connection. Let $q_i$ be the class of $e_i^{[p]}$ in $(F_B^*E/A)_K$ for
$1\leq i\leq n$. Then
$
\beta_K(\ell)=\sum_{i=1}^nq_i\otimes du_i,
$
and hence, for the dual basis $q_i^\vee$, $\theta_K(q_i^\vee\otimes\ell)=du_i.$

By \cite[Proposition~3.2 and Theorem~3.3]{SchroerFibrations}, regularity of
the generic $p$-Fermat hypersurface is equivalent to the coefficient ratios
$u_1,\ldots,u_n$ having $p$-degree $n$ over $K^p$. Thus they are
$p$-independent. The $p$-basis criterion
\cite[Tag 07P2]{StacksProject} implies that
$du_1,\ldots,du_n$ are linearly independent in
$\Omega^1_{K/\mathbb F_p}$. Therefore $\theta_K$ is injective, and
$\theta$ induces an injection
$$
(F_B^*E/A)^\vee\otimes A
\hookrightarrow
\left(\Omega^1_{B/\mathbb F_p}\right)_{\mathrm{tf}}.
$$

We next show that the composite with
$R:=(\Omega^1_{B/k})_{\mathrm{tf}}$ has generic rank one. Since $B$ is
geometrically integral, $K/k$ is regular of transcendence degree one, so
$\dim_K\Omega^1_{K/k}=1$ and the rank is at most one. Suppose that it were
zero. Then $d_{K/k}u_i=0$ for $(1\leq i\leq n).$
By \cite[Tag 07P2]{StacksProject}, $K/k$ has a one element $p$-basis
$t$. Thus $1,t,\ldots,t^{p-1}$ is a basis of $K$ over the compositum
$K^pk$ of $K^p$ and $k$ inside $K$. Differentiating in this basis shows
that an element has zero derivative exactly when only its constant
coefficient remains. Hence
$
\ker(d_{K/k})=K^pk.
$

Hence $u_i\in K^pk$ for every $i$. Since there are only finitely many
coefficients, we may choose $c_1,\ldots,c_r\in k$ such that
$
u_i\in K^p(c_1,\ldots,c_r)
\text{ for }
1\leq i\leq n.
$
Set
$
k':=k(c_1^{1/p},\ldots,c_r^{1/p}).
$
Then $k'/k$ is finite and purely inseparable and
$u_i\in(Kk')^p$ for all $i$.

Geometric integrality of $B$ and $X$ implies that $B_{k'}$ and $X_{k'}$ are
integral. Since the defining equation is nonzero on every fibre, the projection
$X\to B$ is dominant. The generic fibre of $X_{k'}\to B_{k'}$ is therefore
integral and its affine coordinate rings are localizations of domains. Its ground
field is $K(B_{k'})=Kk'$, obtained by extending the constants from $k$ to $k'$.
On the other hand, choosing $a_i\in Kk'$ with $a_i^p=u_i$, its equation is
$$
x_0^p+u_1x_1^p+\cdots+u_nx_n^p
=
\left(x_0+a_1x_1+\cdots+a_nx_n\right)^p,
$$
which is nonreduced. This contradiction shows that the relative component of
$\theta_K$ is nonzero. Its target is one dimensional, so it has rank one.

Proposition~\ref{prop:differential-defect}, applied to
$\mathcal V=(F_B^*E/A )^\vee\otimes A$, now gives
$
\deg(F_B^*E/A )^\vee\otimes A)\leq2g-2.
$
Since
$
\deg F_B^*E/A
=
\deg(F_B^*E)-\deg A
=
p\deg E-\deg A
$
and $\operatorname{rank} (F_B^*E/A )=n$, we have
$$
\deg((F_B^*E/A )^\vee\otimes A)
=
-\deg (F_B^*E/A)+n\deg A
=
(n+1)\deg A-p\deg E.
$$

\end{proof}
\begin{remark}\label{rem:p-fermat-context}
The differential mechanism above is the same one that appears in the classical
smooth base theory of wild hypersurface bundles. In the setup of Mori--Saito,
Sato records a Frobenius line
$\mathcal O_Sj\subset F_S^*E\otimes L$ and a surjection from $T_S$ to its
quotient; in the maximal case this gives an exact sequence
\cite[Definition~3.1, Theorem~3.2, and Definition~3.3]{SatoWild}. Theorem~\ref{thm:p-fermat-degree} instead keeps only the numerical consequence over a
regular curve which need not be smooth over $k$. Geometric integrality of $X$
is used precisely to rule out relative rank zero; it cannot simply be omitted,
as a constant family with regular $p$-Fermat fibre becomes a $p$-fold hyperplane after a
purely inseparable extension of the constants.
\end{remark}

\section{Wild Conic Bundles in Characteristic \texorpdfstring{$2$}{2}}\label{sec:wild-conics}

We now return to the wild conic geometry needed for Theorem~A. Throughout this section, $k$ has characteristic $2$ and $\pi:Z\longrightarrow B$ is a flat contraction from a geometrically integral regular projective surface to a geometrically integral regular projective curve. We assume that every fibre is an irreducible reduced plane conic and that the generic fibre is regular but non smooth. Since $B$ is geometrically integral over $k$, the extension $K(B)/k$ is regular; properness also gives $H^0(B,\mathcal O_B)=k$. We identify the generic conic as anisotropic and quasilinear and show that its relative equation defines the Frobenius line required by Theorem~\ref{thm:p-fermat-degree}. The wild conic estimate is then a direct specialization of that theorem, after which a parity argument sharpens the extremal curve estimate.

\subsection{Quasilinear conics and the Frobenius equation}

We begin with the normal form of the generic fibre. Non smoothness together with regularity forces the polar form to vanish, and regularity then forces the resulting quasilinear form to be anisotropic. We refer \cite{TotaroQuadrics} for the definition of terms like anisotropic, etc used in this section.

\begin{lemma}\label{lem:quasilinear-conic}
Let $K$ be a field of characteristic $2$, and let
$C\subset \mathbb P^2_K$ be a regular integral conic which is not smooth over
$K$. Then $C$ is defined by an anisotropic quasilinear quadratic form.
Equivalently, after choosing coordinates, its equation is a sum of square
terms and has no mixed terms. In fact, any quasilinear plane conic with a
$K$ rational point is not regular at that point.
\end{lemma}

\begin{proof}
This is \cite[Corollary~2.5]{AuelBigazziBohningBothmer}. 
For the final assertion, let $C$ be a quasilinear plane conic and
suppose that $P\in C(K)$. After a change of
coordinates, we may assume that $P=[0:0:1].$
Then the origin in the affine chart $Z=1$ is singular.
\end{proof}

We next globalize the preceding description. In characteristic $2$, the square monomials form the Frobenius pullback $F_B^*E\subset\operatorname{Sym}^2E$, and generic quasilinearity forces the relative conic equation to factor through this subbundle.
In characteristic $2$ there is a canonical Frobenius square subbundle
$F_B^*E\hookrightarrow \operatorname{Sym}^2E.$ When $\operatorname{rank}E=3$, the quotient is locally free
of rank $3$, locally spanned by the mixed monomials
$e_1e_2,e_1e_3,e_2e_3$.

\begin{proposition}\label{prop:wild-equation}
Assume $\operatorname{char}(k)=2$ and the generic fibre $Z_\eta$ of
$\pi:Z\to B$ is a regular but non smooth conic over $K(B)$. Then the
relative conic equation
$
\det E\longrightarrow \operatorname{Sym}^2E
$
factors through the Frobenius square subbundle
$$
F_B^*E\hookrightarrow \operatorname{Sym}^2E.
$$
Moreover, the induced morphism
$
\det E\hookrightarrow F_B^*E
$
is a line subbundle.
\end{proposition}

\begin{proof}
Proposition~\ref{prop:relative-anticanonical-model} identifies the relative
equation with a morphism $\det E\to\operatorname{Sym}^2E$. Its composite
with the locally free quotient $\operatorname{Sym}^2E/F_B^*E$ vanishes
generically by Lemma~\ref{lem:quasilinear-conic}. Since
$(\operatorname{Sym}^2E/F_B^*E)\otimes(\det E)^{-1}$ is torsion free on the integral curve $B$, the
composite vanishes everywhere, giving $s:\det E\to F_B^*E$.

By base change compatibility of the relative anticanonical embedding,
$s\otimes\kappa(b)$ defines the plane conic $Z_b$ for every $b\in B$,
so it is nonzero. Locally, $s$ is therefore a column vector with a unit
entry, and its image is a direct summand. Thus $s$ is a line subbundle.
\end{proof}

The preceding proposition puts the wild conic in the framework of Theorem~\ref{thm:p-fermat-degree}.  Here $p=2$, $n=2$, and the Frobenius line is $A=\det E$.  Thus the general differential estimate of Section~\ref{sec:differentials} already contains the characteristic $2$ conic estimate needed below.

\begin{theorem}\label{thm:wild-conic-estimate}
Let $k$ be a field of characteristic $2$, and let $\pi:Z\longrightarrow B$
be a flat contraction from a geometrically integral regular projective surface to a geometrically integral regular projective curve.  Assume that every fibre is an irreducible reduced plane conic and that the generic fibre is regular but non smooth.  Put
$$
g:=h^1(B,\mathcal O_B),
\qquad
E:=\pi_*\omega^{-1}_{Z/B}.
$$
Then $\deg E=K_{Z/B}^2\leq2g-2$.
Consequently,
$
K_Z^2\leq6-6g.
$
\end{theorem}

\begin{proof}
By Proposition~\ref{prop:wild-equation}, the relative equation defines a line
subbundle
$
A:=\det E\hookrightarrow F_B^*E.
$
The associated degree $2$ hypersurface in $\mathbb P_B(E)$ is precisely $Z$,
which is geometrically integral, and its generic $2$-Fermat conic is regular.
Theorem~\ref{thm:p-fermat-degree}, with $p=2$ and $n=2$, therefore gives
$
3\deg(\det E)-2\deg E\leq2g-2.
$
Since $\deg(\det E)=\deg E$, this is $\deg E\leq2g-2$.
Proposition~\ref{prop:relative-anticanonical-model} gives
$K_{Z/B}^2=\deg E$, and as in the proof of Proposition~\ref{prop:gensmooth-fibre-bound} we have 
$K_Z^2=K_{Z/B}^2+8(1-g)$. Combining all this, we get $K_Z^2=\deg E+8(1-g)\leq6-6g$.
\end{proof}

\subsection{Parity and the genus bound}

It remains to compare the preceding upper bound with the extremal curve lower bound from Section~\ref{sec:reduction}. In the wild case, anisotropy imposes an additional parity condition on every multisection of the generic conic, improving the numerical estimate.

\begin{lemma}\label{lem:wild-multisection-parity}
Let $C$ be an anisotropic quasilinear conic over a field $K$ of characteristic $2$. Every closed point of $C$ has even degree over $K$. In particular, in the wild fibre type situation the integer $m_\Gamma=[K(\Gamma):K(B)]$
is even.
\end{lemma}

\begin{proof}
If a closed point $x\in C$ had odd degree, its residue extension
$\kappa(x)/K$ would be separable, since the inseparable degree is a power
of $2$. By \cite[Lemma~2.1]{TotaroQuadrics}, the defining quasilinear form
remains anisotropic under separable extensions, contradicting the
$\kappa(x)$-rational point. Hence every closed point has even degree.
Applying this to the point of $Z_\eta$ defined by the generic point of
$\Gamma$, whose residue field is $K(\Gamma)$, shows that $m_\Gamma$ is even.
\end{proof}

\begin{proposition}\label{prop:char2-wild}
Retain the fibre type MMP setting of
Subsections~\ref{subsec:regular-mmp}--\ref{subsec:conic-output}. Assume
$\operatorname{char}(k)=2$ and that the generic fibre of $\pi$ is a regular
non smooth conic over $K(B)$. Then
$$
q(Z)
\leq
\max\left\{1,\frac1\varepsilon-\frac13\right\}
\leq
1+\frac1\varepsilon.
$$
If $0<\varepsilon\leq1$, then in particular $q(Z)\leq1/\varepsilon$.
\end{proposition}

\begin{proof}
Set $g:=h^1(B,\mathcal O_B)$.
By Subsection~\ref{subsec:conic-output}, $q(Z)=g$. Theorem~\ref{thm:wild-conic-estimate} gives $K_Z^2\leq6-6g$.
Let $\Gamma$ be the extremal curve of \cite[Lemma~5.5]{BernasconiMartin}, and retain the notation $d_\Gamma,m_\Gamma,n$ from the proof of Proposition~\ref{prop:kz-lower-bound}. Thus
$$
K_Z^2=
\frac{d_\Gamma}{m_\Gamma^2}
\bigl(8m_\Gamma+4n(1-m_\Gamma)\bigr),
\qquad
n\leq\frac2\varepsilon.
$$
Put $\Xi:=8m_\Gamma+4n(1-m_\Gamma)$.
If $\Xi\geq0$, then $K_Z^2\geq0$, and the upper bound $K_Z^2\leq6-6g$ gives $g\leq1$.

Assume $\Xi<0$. The negative case in the proof of
Proposition~\ref{prop:kz-lower-bound} gives $d_\Gamma\leq m_\Gamma\leq5$.
Lemma~\ref{lem:wild-multisection-parity} shows that $m_\Gamma$ is even, so
$m_\Gamma\in\{2,4\}$.
Because $\Xi<0$ and $d_\Gamma\leq m_\Gamma$,
$$
K_Z^2
=
\frac{d_\Gamma}{m_\Gamma^2}\Xi
\geq
\frac1{m_\Gamma}\Xi
=
8-4n\frac{m_\Gamma-1}{m_\Gamma}.
$$
Since $m_\Gamma\leq4$, $K_Z^2\geq8-3n\geq8-\frac6\varepsilon$.
Combining this with $K_Z^2\leq6-6g$ yields $g\leq\frac1\varepsilon-\frac13$.
The two cases give
$$
g\leq\max\left\{1,\frac1\varepsilon-\frac13\right\},
$$
and the remaining assertions follow.
\end{proof}

\section{Proof of the Irregularity Theorem}\label{sec:completion}

We now assemble the preceding estimates. For fibre type outputs, characteristic $2$ was treated in Sections~\ref{sec:relative-conics} and~\ref{sec:wild-conics}; in characteristics $3$ and $5$, the generic conic is smooth by the result of Witaszek recalled below. The remaining MMP outputs are regular del Pezzo surfaces and are controlled by the results of Bernasconi--Martin.

\begin{proof}[Proof of Theorem~\ref{thm:main}]
Subsection~\ref{subsec:regular-mmp} gives $q(X)=q(Z)$, so it is enough to bound $q(Z)$.

Suppose first that $\dim B=1$. If the generic fibre is smooth, then Proposition~\ref{prop:gensmooth-fibre-bound} gives
$$
q(Z)\leq1+\frac{4}{5\varepsilon}\leq1+\frac1\varepsilon,
$$
If the generic fibre is a non-smooth conic, then $\operatorname{char}(k)\neq \{3,5\}$. Indeed, $\eta$ be the generic point of $B$ and write $F_\eta:=Z_\eta$. By the conic bundle description, $F_\eta$ is a regular integral plane conic over $K(B)$. Adjunction gives $\omega_{F_\eta/K(B)}\simeq\mathcal O_{F_\eta}(-1)$, so $\deg K_{F_\eta}=-2$. Moreover, $H^0(F_\eta,\mathcal O_{F_\eta})=K(B)$. Since $\pi$ is a contraction of relative dimension one and $\operatorname{char}(k)>2$, \cite[Proposition~2.2]{WitaszekCBF} applies and shows that $F_\eta$ is smooth over $K(B)$. So, when the generic fibre is a non-smooth conic, then $\operatorname{char}(k)=2$, in which case Proposition~\ref{prop:char2-wild} gives $q(Z)\leq1+\frac1\varepsilon$.

It remains to consider $\dim B=0$. Then $Z$ is a regular geometrically integral del Pezzo surface by Lemma~\ref{lem:geometric-integrality-output}. Its structure morphism to $\Spec k$ is locally complete intersection, since both source and target are regular; see \cite[Tag 0E9K]{StacksProject}. In characteristic $2$, \cite[Theorem~1.2]{BernasconiMartin} gives $q(Z)\leq1$. In characteristic $3$, regularity makes $Z$ canonical, and \cite[Proposition~4.10]{BernasconiMartin} shows that $Z$ is tame in the terminology of \cite{BernasconiMartin}, hence $q(Z)=0$. In characteristic $5$, positive irregularity is excluded by \cite[Theorem~1.2]{BernasconiMartin}, so again $q(Z)=0$. These cases exhaust the MMP outputs. Therefore $q(X)=q(Z)\leq1+\frac1\varepsilon.$
\end{proof}

\section{BAB in All Positive Characteristics}\label{sec:bab}

We combine Theorem~\ref{thm:main} with the numerical bounds of \cite{BernasconiMartin} and the boundedness theorem for polarised surfaces of \cite{KollarModuli}. Bounded irregularity gives finitely many possible Euler characteristics, and hence finitely many Hilbert polynomials. The argument here is essentially the same as Bernasconi-Martin’s argument, and goes back to Koll\'ar.

\begin{theorem}\label{thm:BAB-small-char}
Fix a rational number $\varepsilon>0$ and let $p\in\{2,3,5\}$. Then the class
$$
\mathcal X_{dP,\varepsilon}^{(p)}:=\left\{X \;\middle|\; X \text{ is a geometrically integral }\varepsilon\text{-klt del Pezzo surface over a field of characteristic }p\right\}
$$
is bounded over $\Spec \mathbb Z$.
\end{theorem}

\begin{proof}
Let $X\in\mathcal X_{dP,\varepsilon}^{(p)}$. By Theorem~\ref{thm:main}, $q(X)\leq1+\frac1\varepsilon$. The volume bound \cite[Theorem~5.6]{BernasconiMartin} gives $K_X^2\leq V(\varepsilon)$. By \cite[Lemma~5.9 and Proposition~5.11]{BernasconiMartin}, there is a common integer $n=n(\varepsilon)>0$ such that $-nK_X$ is Cartier. These results apply in every positive characteristic without a tameness assumption. Since $-K_X$ is ample, $K_X^2>0$; as $n^2K_X^2\in\mathbb Z$, only finitely many values of $K_X^2$ can occur. Put $L:=\mathcal O_X(-nK_X)$.

It remains to replace the tame identity used in \cite[Corollary~5.8]{BernasconiMartin} by the conclusion of Theorem~\ref{thm:main}.  Let $\mu:Y\to X$ be the minimal regular resolution.  By Proposition~\ref{prop:klt-rational-sing}, $H^i(X,\mathcal O_X)\simeq H^i(Y,\mathcal O_Y)$ for all $i$.  
Because $Y$ is of del Pezzo type.  In particular, $K_Y$ is not pseudo-effective, so $H^0(Y,\omega_Y)=0$ and hence $H^2(X,\mathcal O_X)=0$ by Serre duality.  Thus
$
\chi(\mathcal O_X)=1-q(X),
$
which ranges over a finite set depending only on $\varepsilon$. Rationality of the singularities and the projection formula identify the Euler characteristic of a Cartier divisor on $X$ with that of its pullback to $Y$. The pullback has zero intersection with every exceptional curve. Riemann--Roch on $Y$ therefore gives
$$
P_X(t):=\chi\bigl(X,L^{\otimes t}\bigr)
=
\chi(\mathcal O_X)
+
\frac{nt(nt+1)}2K_X^2.
$$
Thus the polynomials $P_X$ form a finite set.

Let $\bar k$ be an algebraic closure of $k$. The surface $X_{\bar k}$ is integral, and the ample line bundle $L_{\bar k}$ has Hilbert polynomial $P_X$, since coherent cohomology commutes with field extension. By \cite[Definition~2.1.1 and Theorem~2.1.2]{KollarModuli}, polarized reduced projective surfaces with a fixed Hilbert polynomial are bounded in arbitrary characteristic. Applying this result to the finitely many polynomials above gives an integer $m>0$, independent of $X$ and $k$, such that $L_{\bar k}^{\otimes m}$ is very ample and its space of sections has uniformly bounded dimension. Here the polarization is the pullback of $L$; geometric normality of $X$ is not required.

Flat base change identifies
$$
H^0(X,L^{\otimes m})\otimes_k\bar k
\simeq H^0(X_{\bar k},L_{\bar k}^{\otimes m}).
$$
The evaluation map for $L^{\otimes m}$ is therefore surjective, and the resulting morphism to the projective space of its complete linear system becomes a closed immersion over $\bar k$. Faithfully flat descent makes it a closed immersion over $k$; see \cite[Tag 02KH and Tag 02L6]{StacksProject}. Consequently, every $X$ embeds into a fixed projective space $\mathbb P^N_k$, with Hilbert polynomial belonging to the finite set $\{P_X(mt)\}$. The universal families over the corresponding Hilbert schemes of $\mathbb P^N_{\mathbb F_p}$ give a projective flat family of finite type over $\Spec\mathbb Z$ containing every such $X$.
\end{proof}

\enlargethispage{3\baselineskip}
\begin{corollary}\label{cor:all-positive-char-bab}
Fix a rational number $\varepsilon>0$.  The class of geometrically integral $\varepsilon$-klt del Pezzo surfaces over fields of positive characteristic is bounded over $\Spec\mathbb Z$.
\end{corollary}

\begin{proof}
Theorem~\ref{thm:BAB-small-char} proves the assertion in characteristics $2$, $3$, and $5$.  In every other positive characteristic it is exactly the boundedness theorem of Bernasconi--Martin; see \cite[Theorem~5.12]{BernasconiMartin}.
\end{proof}

\section*{Acknowledgment}
We thank Joe Waldron and Omprokash Das for fruitful discussions,  insightful advice, and their immense support. This project began during the 2025 Summer Research Institute in Algebraic Geometry at Colorado State University. We thank the university and the organisers for their hospitality.  We are grateful to Fabio Bernasconi for fruitful discussion at Notions of Singularity in
Different Characteristics workshop at Banff International Research Station. The first author is grateful to the BIRS and the organisers for their hospitality. Gemini Pro and ChatGPT were used to find relevant references and to improve the readability of some sentences throughout the article. The first author was supported by NSF Grant DMS \#2401279 and by the Simons Foundation (award \#850684, JW) during the preparation of this paper.

\renewcommand{\footnotesize}{\normalfont\normalsize}
\end{document}